\documentclass[psamsfonts]{amsart}

\usepackage{amssymb, graphicx, tikz, very-funny-picture}
\newtheorem{theorem}{Theorem}
\newtheorem{lemma}{Lemma}

\newcommand{\del}{\partial}
\DeclareMathOperator{\Vol}{Vol}
\newcommand{\geqnbd}{\geq_{\mathrm{nbd}}}

\begin{document}

\title[Maps with small fibers]{A family of maps with many small fibers}
\author{Hannah Alpert}
\author{Larry Guth}
\address{MIT\\ Cambridge, MA 02139 USA}
\email{hcalpert@math.mit.edu}
\email{lguth@math.mit.edu}
\subjclass[2010]{53C23}
\begin{abstract}
The waist inequality states that for a continuous map from $S^n$ to $\mathbb{R}^q$, not all fibers can have small $(n-q)$-dimensional volume.  We construct maps for which most fibers have small $(n-q)$-dimensional volume and all fibers have bounded $(n-q)$-dimensional volume.
\end{abstract}
\maketitle

Let $n, q \in \mathbb{N}$ with $n > q \geq 1$, and let $f : S^n \rightarrow \mathbb{R}^q$ be a continuous map.  Let $\widehat{p} : \mathbb{R}^{n+1} \rightarrow \mathbb{R}^q$ be a surjective linear map, and let $p = \widehat{p}\vert_{S^n}$.  The waist inequality states that the largest fiber of $f$ is at least as large as the largest fiber of $p$:
\[\sup_{y \in \mathbb{R}^q} \Vol_{n-q} f^{-1}(y) \geq \sup_{y \in \mathbb{R}^q} \Vol_{n-q} p^{-1}(y).\]
See \cite{almgren65}, \cite{gromov83}, \cite{gromov03}, and \cite{memarian11} for proofs of the waist inequality, or \cite{guth14} for a survey.  In the case $q = 1$, the waist inequality is a consequence of the isoperimetric inequality on $S^n$.  The isoperimetric inequality can also be used to prove that the portion of $S^n$ covered by small fibers of $f$ is not very big; that is, for all $\varepsilon$, we have
\[\Vol_n\ f^{-1}\{y : \Vol_{n-q} f^{-1}(y) < \varepsilon\} \leq \Vol_n\ p^{-1}\{y : \Vol_{n-q} p^{-1}(y) < \varepsilon\}.\]
The theorem presented in this paper describes how the same statement does not hold in the case $q > 1$.  We have also included an appendix with a more precise statement of the waist inequality and the isoperimetric inequality.

\begin{theorem}\label{main-thm}
For every $n, q \in \mathbb{N}$ with $n > q > 1$, and for every $\varepsilon > 0$, there is a continuous map $f: S^n \rightarrow \mathbb{R}^q$ such that all but $\varepsilon$ of the $n$-dimensional volume of $S^n$ is covered by fibers that have $(n-q)$-dimensional volume at most $\varepsilon$.  Moreover, we may require that every fiber of $f$ has $(n-q)$-dimensional volume bounded by $C_{n, q}$, a constant not depending on $\varepsilon$.
\end{theorem}

In what follows, $I^n = [0, 1]^n$ denotes the $n$-dimensional unit cube, and $\del I^n$ denotes its boundary.  A \emph{tree} refers to the topological space corresponding to a graph-theoretic tree: topologically, a tree is a finite $1$-dimensional simplicial complex that is contractible.  

The bulk of the construction comes from the following lemma, in which we construct a preliminary ``tree map'' $t_{n, r, \delta}$ from $I^n$ to a tree.  Later, to construct $f$ we will change the domain from $I^n$ to $S^n$ by gluing several tree maps together, and we will change the range from the tree to $\mathbb{R}^q$ by composing with a map from a thickened tree to $\mathbb{R}^q$.  In the tree map $t_{n, r, \delta}$, the parameter $r$ corresponds to the depth of the tree.  As $r$ increases, the typical fiber of the map becomes smaller.  The parameter $\delta$ corresponds to the total volume of the larger fibers.

\begin{lemma}
For every $n, r \in \mathbb{N}$, there is a rooted tree $T_{n, r}$ such that for every $\delta > 0$ there is a continuous map $t_{n, r, \delta} : I^n \rightarrow T_{n, r}$ with the following properties:
\begin{enumerate}
\item Every fiber of $t_{n, r, \delta}$ is either a single point, the boundary of an $n$-dimensional cube of side length at most $1$, or the $(n-1)$-skeleton of a $2 \times 2 \times \cdots \times 2$ array of $n$-dimensional cubes each of side length at most $\frac{1}{2}$.
\item All but $\delta$ of the volume of $I^n$ is covered by fibers of $t_{n, r, \delta}$ that are boundaries of $n$-dimensional cubes of side length at most $2^{-r}$.
\item $t_{n, r, \delta}(\del I^n)$ is a single point, the root of $T_{n, r}$.
\item Each vertex has at most $2^n$ daughter vertices.
\end{enumerate}
\end{lemma}

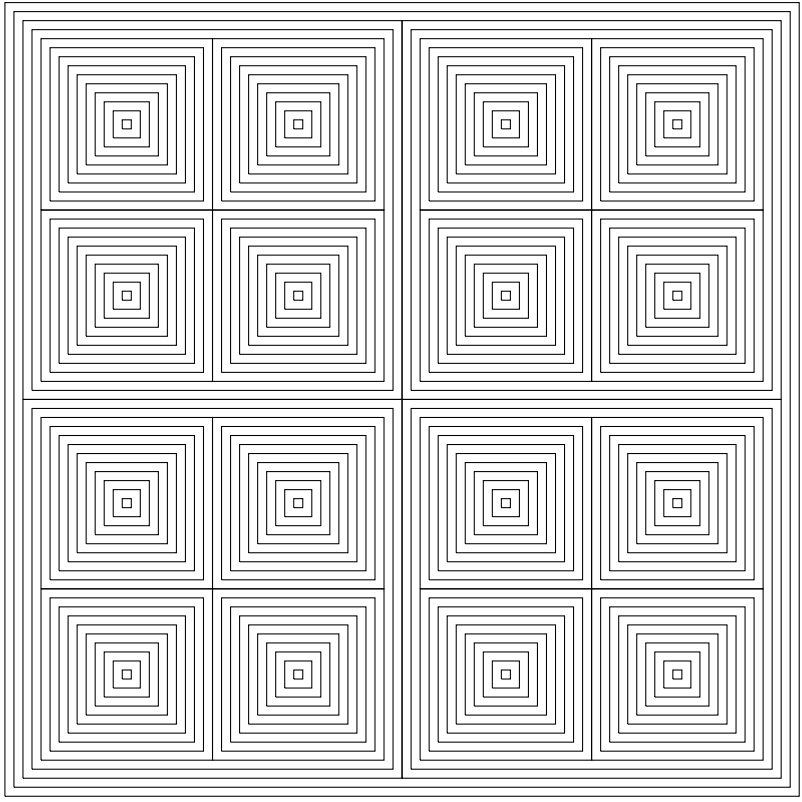
\begin{figure}
\begin{center}
\begin{tikzpicture}[scale=.12]
\veryfunnypicture
\end{tikzpicture}
\end{center}
\caption{Every fiber of $t_{2, 2, \delta}$ has length at most 6, and most fibers have length at most 1.}
\end{figure}

\begin{proof}
We construct the tree and tree map recursively in $r$.  For $r = 0$, the tree $T_{n, 0}$ is a single edge which we may identify with the interval $[0, \frac{1}{2}]$, with $0$ being the root.  For any $\delta$, we set $t_{n, 0, \delta}(x) = \mathrm{dist}(x, \del I^n)$ for all $x \in I^n$.

Now let $r > 0$.  To construct $T_{n, r}$, we take the disjoint union of one copy of $[0, 1]$ and $2^n$ copies of $T_{n, r-1}$, and identify the root of every copy of $T_{n, r-1}$ with $1 \in [0, 1]$.  The root of $T_{n, r}$ is $0 \in [0, 1]$.  We define $t_{n, r, \delta}$ piecewise as follows.  For some small choice of $\delta_1 > 0$, we define $t_{n, r, \delta}$ on the closed $\delta_1$-neighborhood of $\del I^n$ to $[0, 1] \subset T_{n, r}$ by 
\[t_{n, r, \delta}(x) = \frac{1}{\delta_1} \mathrm{dist}(x, \del I^n).\]
Then, translating the coordinate hyperplanes to pass through the center of $I^n$ we divide the remainder of the cube into a $2 \times 2 \times \cdots \times 2$ array of cubes $Q_1, \ldots, Q_{2^n}$ each of side length slightly less than $\frac{1}{2}$.  For each $j = 1, \ldots 2^n$, let $\lambda_j: Q_j \rightarrow I^n$ be the map that scales $Q_j$ up to unit size, and let $i_j: T_{n, r-1} \rightarrow T_{n, r}$ be the inclusion of the $j$th copy of $T_{n, r-1}$ into $T_{n, r}$.  Then for some small choice of $\delta_2 > 0$, we put
\[t_{n, r, \delta}\vert_{Q_j} = i_j \circ t_{n, r-1, \delta_2} \circ \lambda_j.\]

Properties 1, 3, and 4 are easily satisfied by the construction.  To ensure property 2, we need to choose $\delta_1$ and $\delta_2$.  The volume of $I^n$ that is covered by large fibers---fibers not equal to the boundary of a cube of side length at most $2^{-r}$---is at most $\delta_1 \cdot 2n + 2^n \cdot \delta_2 \cdot 2^{-n}$, because the area of $\del I^n$ is $2n$ and because the portion of each $Q_j$ that is covered by large fibers has volume at most $\delta_2 \cdot \Vol(Q_j) < \delta_2 \cdot 2^{-n}$.  Thus we may choose $\delta_1 = \frac{\delta}{4n}$ and $\delta_2 = \frac{\delta}{2}$.
\end{proof}

\begin{proof}[Proof of Theorem~\ref{main-thm}]
We may replace $S^n$ by $\del I^{n+1}$ by composing with the (bi-Lipschitz) homeomorphism $\psi: S^n \rightarrow \del I^{n+1}$ given by lining up the centers of $S^n$ and $\del I^{n+1}$ in $\mathbb{R}^{n+1}$ and projecting radially.  We start by constructing a tree $T$ and a tree map $t: \del I^{n+1} \rightarrow T$.  For some large choice of $r$, let $T$ be the tree obtained by identifying the roots of $2(n+1)$ copies of $T_{n, r}$, one for each $n$-dimensional face of $\del I^{n+1}$.  For some small choice of $\delta$, define $t$ on each $n$-dimensional face of $\del I^{n+1}$ to be the composition of $t_{n, r, \delta}$ with the inclusion of the corresponding $T_{n, r}$ into $T$.

The fibers of $t$ have dimension $n-1$.  In order to cut the fibers down to dimension $n-q$, we next construct a projection map $p: \del I^{n+1} \rightarrow \mathbb{R}^{q-1}$ such that the fibers of $p$ intersect the fibers of $t$ transversely.  The fibers of $t$ have codimension $2$ in $\mathbb{R}^{n+1}$ and are aligned with the standard coordinates, so we achieve transversality by using other linear coordinates to construct $p$.  We choose $q-1$ linearly independent vectors $v_1, \ldots, v_{q-1} \in \mathbb{R}^{n+1}$ such that for every two standard basis vectors $e_i, e_j \in \mathbb{R}^{n+1}$ the spaces $\mathrm{span}\{e_i, e_j\}^\perp$ and $\mathrm{span}\{v_1, \ldots, v_{q-1}\}^\perp$ intersect transversely; equivalently, the set $e_i, e_j, v_1, \ldots, v_{q-1}$ is linearly independent.  For $k = 1, \ldots, q-1$, define the $k$th component of $p$ to be the dot product of the input with $v_k$.  Then the fibers of $t\times p: \del I^{n+1} \rightarrow T \times \mathbb{R}^{q-1}$ are codimension $q-1$ transverse linear cross-sections of the $(n-1)$-dimensional fibers of $t$, and have $(n-q)$-dimensional volume bounded by some constant depending on $n$ and $q$.

There exists $M$ large enough that $p(\del I^{n+1})$ is contained in the $(q-1)$-dimensional ball $B(M)$ of radius $M$.  We define a map $\phi: T \times B(M) \rightarrow \mathbb{R}^q$ such that the number of points in each fiber of $\phi$ is at most the maximum degree of $T$, which is $2^n + 1$.  Then we define $f = \phi \circ (t\times p)$.  The fibers of $f$, like the fibers of $t \times p$, have $(n-q)$-dimensional volume bounded by a constant $C_{n, q}$.

The map $\phi$ is constructed as follows.  Let $\phi\vert_{T \times \{0\}}$ be an embedding of $T$ into $\mathbb{R}^q$ in which the edges map to straight line segments and each daughter vertex has $x_1$-coordinate greater than that of its parent.  Let $d$ be the minimum distance between disjoint edges of $\phi(T\times \{0\})$.  Then for every $p \in T$ and $x \in B(M)$, we set
\[\phi(p, x) = \phi(p, 0) + \frac{d}{4} \left(0, \frac{x}{M}\right),\]
where $\left(0, \frac{x}{M}\right)$ denotes the point in $\mathbb{R}^q$ constructed by adding onto $\frac{x}{M} \in \mathbb{R}^{q-1}$ a first coordinate of $0$.  If $\phi(p, x) = \phi(p', x')$, then  $\phi(p, 0)$ and $\phi(p', 0)$ are at most $\frac{d}{2}$ apart, so $p$ and $p'$ lie on two incident edges of $T$;  also, $\phi(p, 0)$ and $\phi(p', 0)$ have the same $x_1$-coordinate, so these two edges are between two daughters and a common parent, rather than a daughter, a parent, and a grandparent.

To finish the proof, we show that $\delta$ and $r$ may be chosen such that all  but $\varepsilon$ of the $n$-dimensional volume of $\del I^{n+1}$ is covered by fibers with $(n-q)$-dimensional volume at most $\varepsilon$.  The maximum number of daughter vertices of any vertex of $T$ is $2^n$, and most of $\del I^{n+1}$ is covered by fibers of $f$ that are unions of at most $2^n$ codimension $q-1$ transverse linear cross-sections of boundaries of $n$-dimensional cubes of side length at most $2^{-r}$.  We choose $r$ large enough that every codimension $q-1$ transverse linear cross-section of $2^{-r} \del I^n$ has $(n-q)$-dimensional volume at most $\frac{\varepsilon}{2^n}$.  The volume of the portion of $\del I^{n+1}$ covered by larger fibers is at most $2(n+1) \cdot \delta$, so we choose $\delta < \frac{\varepsilon}{2(n+1)}$.
\end{proof}

\section*{Appendix: The waist inequality and the isoperimetric inequality}

In order to be precise about the waist inequality, we need a notion of $(n-q)$-dimensional volume of arbitrary closed subsets in $S^n$.  Gromov's version of the waist inequality is stated in terms of the Lebesgue measures $\Vol_n$ of the $\varepsilon$-neighborhoods $f^{-1}(y)_\varepsilon$ of the fibers $f^{-1}(y)$ of a continuous map $f$.

\begin{theorem}[Waist inequality, \cite{gromov03}]
Let $f: S^n \rightarrow \mathbb{R}^q$ be a continuous map.  Then there exists a point $y \in \mathbb{R}^q$ such that for all $\varepsilon > 0$, we have
\[\Vol_n(f^{-1}(y)_\varepsilon) \geq \Vol_n(S^{n-q}_\varepsilon),\]
where $S^{n-q} \subset S^n$ denotes an equatorial $(n-q)$-sphere.
\end{theorem}

The paper \cite{memarian11} gives a detailed exposition of the proof of the waist inequality and fills in some small gaps in the original argument.  For convenience we introduce a notation for comparing the $\varepsilon$-neighborhoods of two sets: given $E, F \subseteq S^n$, we say that $E$ is \textbf{\textit{larger in neighborhood}} than $F$, denoted $E \geqnbd F$, if for all $\varepsilon > 0$ we have
\[\Vol_n(E_\varepsilon) \geq \Vol_n(F_\varepsilon).\]
Then the waist inequality states that for some $y \in \mathbb{R}^q$ we have $f^{-1}(y) \geqnbd S^{n-q}$.

In the case $q = 1$, we would like to say that the waist inequality is a consequence of the isoperimetric inequality.  The classical isoperimetric inequality applies only to regions with smooth boundary, so we need the following version, which is stated and proved in \cite{flm} and attributed to \cite{schmidt48}:

\begin{theorem}[Isoperimetric inequality]
Let $A \subseteq S^n$ be a closed set and $B \subseteq S^n$ be a closed ball with $\Vol_n(B) = \Vol_n(A)$.  Then we have
\[A \geqnbd B.\]
\end{theorem}

In the introduction we claimed that in the case $q = 1$, the isoperimetric inequality could be used to prove, in addition to the waist inequality, another statement about the volume of $S^n$ covered by small fibers.  Here we formulate the statement more precisely and prove it.  The proof implies the waist inequality for $q = 1$.

\begin{theorem}\label{codim1}
Let $f: S^n \rightarrow \mathbb{R}$ be a continuous map, and $p: S^n \rightarrow \mathbb{R}$ be the restriction to $S^n$ of a surjective linear map $\widehat{p}: \mathbb{R}^{n+1} \rightarrow \mathbb{R}$.  Then for all $y \in p(S^n)$, we have
\[\Vol_n\{x \in S^n : f^{-1}(f(x)) \geqnbd p^{-1}(y)\} \geq \Vol_n\{x \in S^n : p^{-1}(p(x)) \geqnbd p^{-1}(y)\}.\]
\end{theorem}

The proof of this theorem is based on the following lemma:

\begin{lemma}
Let $X, Y \subset S^n$ be closed sets with $X \cup Y = S^n$.  Let $B^X, B^Y \subset S^n$ be closed balls such that their two centers are antipodal in $S^n$ and $\Vol_n(B^X) = \Vol_n(X)$ and $\Vol_n(B^Y) = \Vol_n(Y)$.  Then we have
\[X \cap Y \geqnbd B^X \cap B^Y.\]
\end{lemma}

\begin{proof}
First we claim that $(X \cap Y)_\varepsilon$ is the disjoint union of $X_\varepsilon \setminus X$, $Y_\varepsilon \setminus Y$, and $X \cap Y$.  It is clear that $(X \cap Y)_\varepsilon$ is the disjoint union of its intersections with $S^n \setminus X$, $S^n \setminus Y$, and $X \cap Y$.  Thus it suffices to show that
\[(X \cap Y)_\varepsilon \cap (S^n \setminus X) = X_\varepsilon \setminus X.\]
Because $(X \cap Y)_\varepsilon \subseteq X_\varepsilon$, we immediately have
\[(X \cap Y)_\varepsilon \cap (S^n \setminus X) \subseteq X_\varepsilon \setminus X.\]
For the reverse inclusion, let $y \in X_\varepsilon \setminus X$, and let $\gamma: [0, 1] \rightarrow S^n$ be a curve of length at most $\varepsilon$ with $\gamma(0) = y$ and $\gamma(1) = x \in X$.  Let $t \in [0, 1]$ be the greatest value with $\gamma(t) \in Y$.  Then $\gamma(t) \in X \cap Y$, so $y \in (X \cap Y)_\varepsilon$.

Thus, applying the isoperimetric inequality and additivity of measure, we have
\[\Vol_n((X \cap Y)_\varepsilon) = \Vol_n(X_\varepsilon) - \Vol_n(X) + \Vol_n(Y_\varepsilon) - \Vol_n(Y) + \Vol_n(X \cap Y) \geq\]
\vspace{-25pt}
\begin{eqnarray*}
&\geq& \Vol_n(B^X_\varepsilon) - \Vol_n(B^X) + \Vol_n(B^Y_\varepsilon) - \Vol_n(B^Y) + \Vol_n(B^X \cap B^Y) =\\
&=& \Vol_n((B^X \cap B^Y)_\varepsilon).
\end{eqnarray*}
\end{proof}

\begin{proof}[Proof of Theorem~\ref{codim1}]
Without loss of generality we assume $p(S^n) = [0, 1]$ and $y \leq \frac{1}{2}$.  Then on the right-hand side of the desired inequality we have
\[\{x \in S^n : p^{-1}(p(x)) \geqnbd p^{-1}(y)\} = p^{-1}[y, 1-y].\]
Define $\alpha, \beta \in \mathbb{R}$ as
\[\alpha = \sup \{t \in \mathbb{R} : \Vol_n f^{-1}(-\infty, t) \leq \Vol_n p^{-1}[0, y]\},\]
\[\beta = \inf \{t \in \mathbb{R} : \Vol_n f^{-1}(t, \infty) \leq \Vol_n p^{-1}[y, 1]\}.\]
For each $t \in [\alpha, \beta]$, apply the lemma with $X = f^{-1}(-\infty, t]$ and $Y = f^{-1}[t, \infty)$ to get $f^{-1}(t) \geqnbd p^{-1}[y_1, y_2]$ for some $y_1, y_2 \in [y, 1-y]$.  In particular, we have
\[f^{-1}(t) \geqnbd p^{-1}(y_1) \geqnbd p^{-1}(y).\]
Thus, we have
\[f^{-1}[\alpha, \beta] \subseteq \{x \in S^n : f^{-1}(f(x)) \geqnbd p^{-1}(y)\}.\]
Because $\Vol_n f^{-1}(-\infty, \alpha) \leq \Vol_n p^{-1}[0, y]$ and $\Vol_n f^{-1}(\beta, \infty) \leq \Vol_n p^{-1}[y, 1]$ we have
\[\Vol_n f^{-1}[\alpha, \beta] \geq \Vol_n p^{-1}[y, 1-y].\]
\end{proof}

\bibliography{very-funny-map}{}

\providecommand{\bysame}{\leavevmode\hbox to3em{\hrulefill}\thinspace}
\providecommand{\MR}{\relax\ifhmode\unskip\space\fi MR }
\providecommand{\MRhref}[2]{%
  \href{http://www.ams.org/mathscinet-getitem?mr=#1}{#2}
}
\providecommand{\href}[2]{#2}
\begin{thebibliography}{1}

\bibitem{almgren65}
F.J. Almgren, \emph{The theory of varifolds --- a variational calculus in the
  large for the $k$-dimensional area integrated}, Mimeographed notes, 1965.

\bibitem{flm}
T.~Figiel, J.~Lindenstrauss, and V.~D. Milman, \emph{The dimension of almost
  spherical sections of convex bodies}, Acta Math. \textbf{139} (1977),
  no.~1-2, 53--94. \MR{0445274 (56 \#3618)}

\bibitem{gromov83}
M.~Gromov, \emph{Filling {R}iemannian manifolds}, J. Differential Geom.
  \textbf{18} (1983), no.~1, 1--147. \MR{697984 (85h:53029)}

\bibitem{gromov03}
\bysame, \emph{Isoperimetry of waists and concentration of maps}, Geom. Funct.
  Anal. \textbf{13} (2003), no.~1, 178--215. \MR{1978494 (2004m:53073)}

\bibitem{guth14}
L.~Guth, \emph{The waist inequality in {G}romov's work}, The {A}bel {P}rize
  2008--2012 (H.~Holden and R.~Piene, eds.), Springer, 2014, pp.~181--195.

\bibitem{memarian11}
Y.~Memarian, \emph{On {G}romov's waist of the sphere theorem}, J. Topol. Anal.
  \textbf{3} (2011), no.~1, 7--36. \MR{2784762 (2012g:53066)}

\bibitem{schmidt48}
E.~Schmidt, \emph{Die {B}runn-{M}inkowskische {U}ngleichung und ihr
  {S}piegelbild sowie die isoperimetrische {E}igenschaft der {K}ugel in der
  euklidischen und nichteuklidischen {G}eometrie. {I}}, Math. Nachr. \textbf{1}
  (1948), 81--157. \MR{0028600 (10,471d)}

\end{thebibliography}
\bibliographystyle{amsplain}
\end{document}